\documentclass[a4paper,10pt]{amsart}
\usepackage[english]{babel}
\usepackage{amsmath,amssymb,amsthm}
\usepackage{multicol}

\usepackage[pdftex]{color}
\usepackage[bookmarks=true,hyperindex,pdftex,colorlinks, citecolor=blue]{hyperref}

\theoremstyle{plain}
\newtheorem{theorem}{Theorem}[section]
\newtheorem{corollary}[theorem]{Corollary}

\newtheorem{proposition}[theorem]{Proposition}
\newtheorem{lemma}[theorem]{Lemma}

\theoremstyle{definition}

\newtheorem{remark}[theorem]{Remark}
\newtheorem{example}[theorem]{Example}

\newtheorem{definition}[theorem]{Definition}

 \DeclareMathOperator{\re}{Re\,}
 
 \DeclareMathOperator{\Id}{\mathrm{Id}}

 \DeclareMathOperator{\id}{\mathrm{id}}
 \newcommand{\ec}{\mathrm{co}}

\newcommand{\K}{\mathbb{K}}

\newcommand{\C}{\mathbb{C}}
\newcommand{\R}{\mathbb{R}}
\newcommand{\N}{\mathbb{N}}

\newcommand{\eps}{\varepsilon}

\newcommand{\Wa}{\widetilde{W}}
\newcommand{\ixy}{i_{X,Y}}
\newcommand{\iyy}{i_{Y,Y}}

\renewcommand{\leq}{\leqslant}
\renewcommand{\geq}{\geqslant}

\begin{document}
\title{On different definitions of numerical range}

\author{Miguel Mart\'{\i}n}

\address{Departamento de An\'{a}lisis Matem\'{a}tico \\ Facultad de
 Ciencias \\ Universidad de Granada \\ 18071 Granada, Spain
\newline
\href{http://orcid.org/0000-0003-4502-798X}{ORCID: \texttt{0000-0003-4502-798X} }
 }
\email{mmartins@ugr.es}

\date{February 15th, 2015}

\thanks{Supported by Spanish MICINN and FEDER project no.\ MTM2012-31755, and by Junta de Andaluc\'{\i}a and FEDER grants FQM-185 and P09-FQM-4911.}


\begin{abstract}
We study the relation between the intrinsic and the spatial numerical ranges with the recently introduced ``approximated'' spatial numerical range. As main result, we show that the intrinsic numerical range always coincides with the convex hull of the approximated spatial numerical range. Besides, we show sufficient conditions and necessary conditions to assure that the approximated spatial numerical range coincides with the closure of the spatial numerical range.
\end{abstract}

\maketitle
\thispagestyle{empty}

\section{Introduction}
The concept of numerical range of an operator goes back to O.~Toeplitz, who defined in 1918 the field of values of a matrix, a concept easily extensible to bounded linear operators on a Hilbert space. In the 1950's, a concept of numerical range of elements of unital Banach algebras was used to relate the geometrical and algebraic properties of the unit, starting with a paper by H.~Bohnenblust and S.~Karlin where it is shown that the unit is a vertex of the unit ball of the algebra, and was also used in the developing of Vidav's characterization of $C^*$-algebras. Later on, in the 1960's, G.~Lumer and F.~Bauer gave independent but related extensions of Toeplitz's numerical range to bounded linear operators on Banach spaces which do not use the algebraic structure of the space of all bounded linear operators. We refer the reader to the monographs by F.~Bonsall and J.~Duncan \cite{B-D1,B-D2} and to sections \S2.1 and \S2.9 of the very recent book \cite{Cabrera-Rodriguez} by M.~Cabrera and A.~Rodr\'{\i}guez-Palacios for more information and background. Let us present the necessary definitions and notation. We will work with both real and complex Banach spaces. We write $\K$ to denote the base field ($=\R$ or $\C$) and $\re(\cdot)$ to denote the real part in the complex case and just the identity in the real case. Given a Banach space $X$, $S_X$ is its unit sphere, $X^*$ is the topological dual space of $X$ and $L(X)$ is the Banach algebra of all bounded linear operators on $X$. The \emph{intrinsic numerical range} (or \emph{algebra numerical range}) of $T\in L(X)$ is
$$
V(T):=\bigl\{\Phi(T)\,:\,\Phi\in L(X)^*,\,\|\Phi\|=\Phi(\Id)=1 \bigr\},
$$
where $\Id$ denotes the identity operator on $X$. The
\emph{spatial numerical range} of $T$ is given by
$$
W(T):=\bigl\{x^*(Tx)\,:\, x\in S_{X},\,x^*\in S_{X^*},\, x^*(x)=1\bigr\}.
$$
These two ranges coincide in the case when $X$ is a Hilbert space. For arbitrary Banach spaces, the equality
\begin{equation*}%\label{eq:ecc-operator}\tag{$\star$}
\overline{\ec\,W(T)}=V(T)
\end{equation*}
is valid for all $T\in L(X)$ ($\ec\, A$ denotes the convex hull of a set $A$). This equality allows to study algebra numerical ranges of operators without taking into account elements of the (wild) topological dual of the space of operators and, conversely, to get easier proofs of results on spatial numerical ranges.

The above equality has been extended to more general setting, as bounded uniformly continuous functions from the unit sphere of a Banach space to the space \cite{Harris,Rodriguez-JMAA}, but it is known that it is not possible to be extended to all bounded functions \cite{Rodriguez-PAMS}. On the other hand, it is possible to define numerical ranges of operators (or functions) with respect to a fixed operator (or function) which plays the rolle of the identity operator, as it is done in \cite{Harris}. For the intrinsic numerical range, the definition is immediate (see Definition~\ref{def:intrinisic}), but the case of the spatial numerical range (see Definition~\ref{def:spatial}) is more delicate, as we may produce empty numerical ranges. Very recently, a definition of an ``approximated'' spatial numerical range has been introduced \cite{Ardalani} for operators between different Banach spaces, which can be easily extended to bounded functions (see Definition~\ref{def:approximate-spatial}) and which is never empty.
Let us present de definitions of numerical ranges that we will use in the paper in full generality, that is, for bounded functions from a non-empty set into a Banach space. Given a Banach space $Y$ and a non-empty set $\Gamma$, we write $\ell_\infty(\Gamma,Y)$ to denote the Banach space of all bounded functions from $\Gamma$ into $Y$ endowed with the sumpremum norm.

The first definition is the so-called intrinsic numerical range (with respect to a fix function) which appeared, in different settings, with many names and many notations, since the 1960's (holomorphic functions \cite{Harris-holomorphic}, bounded uniformly continuous functions \cite{Harris}, bounded linear operators \cite{B-D1}, bounded functions \cite{Rodriguez-PAMS,Rodriguez-JMAA}, among others). Also, it is nothing but a particular case of the so-called numerical range spaces (see \cite{M-M-P-R} or \cite[\S 2.1 and \S 2.9]{Cabrera-Rodriguez}).

\begin{definition}[Intrinsic numerical range]\label{def:intrinisic}
Let $Y$ be a Banach space and let $\Gamma$ be a non-empty set. We fix $g\in \ell_\infty(\Gamma,Y)$ with $\|g\|=1$. For every $f\in \ell_\infty(\Gamma,Y)$, the \emph{intrinsic numerical range} of $f$ relative to $g$ is
$$
V_g(f):= \bigl\{\Phi(f)\,:\, \Phi\in \ell_\infty(\Gamma,Y)^*,\, \|\Phi\|=\Phi(g)=1\bigr\}.
$$
\end{definition}

Observe that if $\mathcal{M}$ is a closed subspace of $\ell_\infty(\Gamma,Y)$ containing $g$ and $f$, then
$V_g(f)$ can be calculated using only elements in the dual of $\mathcal{M}$ (by Hahn-Banach theorem), so it only depends on the geometry around $f$ and $g$. This is why this numerical range is called ``intrinsic''. Let us also observe that the intrinsic numerical range is a compact and convex subset of $\K$.

The second numerical range we will deal with is the spatial numerical range, which extends the corresponding definition for bounded linear operators.

\begin{definition}[Spatial numerical range]\label{def:spatial}
Let $Y$ be a Banach space and let $\Gamma$ be a non-empty set. We fix $g\in \ell_\infty(\Gamma,Y)$ with $\|g\|=1$. For every $f\in \ell_\infty(\Gamma,Y)$, the \emph{spatial numerical range} of $f$ relative to $g$ is given by
$$
W_g(f):= \bigl\{y^*(f(t))\,:\, y^*\in S_{Y^*},\, t\in \Gamma,\, y^*(g(t))=1\bigr\}.
$$
\end{definition}

Let us observe that $W_g(f)$ is not empty if only if $g(\Gamma)$ intersects $S_Y$. This concept appeared for uniformly continuous functions $g$ and $f$ in a paper by L.~Harris \cite{Harris}. It also has been studied for particular cases of the function $g$, as the inclusion from $S_Y$ into $Y$ \cite{Rodriguez-JMAA,Rodriguez-PAMS} or, more generally, the inclusion from the unit sphere of a subspace into $Y$ \cite{MarMerPay}.

Finally, the last definition is the one given very recently by M.~Ardalani \cite{Ardalani} for bounded linear operators between Banach spaces, which can be extended to arbitrary bounded functions.

\begin{definition}[Approximated spatial numerical range \cite{Ardalani}]\label{def:approximate-spatial}
Let $Y$ be a Banach space and let $\Gamma$ be a non-empty set. We fix $g\in \ell_\infty(\Gamma,Y)$ with $\|g\|=1$. For every $f\in \ell_\infty(\Gamma,Y)$ the \emph{approximated spatial numerical range} of $f$ relative to $g$ is
$$
\Wa_g(f):= \bigcap_{\eps>0} \overline{\bigl\{y^*(f(t))\,:\, y^*\in S_{Y^*},\, t\in \Gamma,\, \re y^*(g(t))>1-\eps\bigr\}}.
$$
\end{definition}

Observe that $\Wa_g(f)$ is a non-empty compact subset of $\K$.

The relation between the intrinsic and the spatial numerical ranges has been studied in the journal literature. Let us present some examples. Given a Banach space $Y$ and a closed subspace $X$ of $Y$, we write $\ixy$ to denote the inclusion from $S_X$ into $Y$, and we just write $\id=\iyy$. It was shown in \cite{Harris} that
$$
V_{\id}(f)= \overline{\ec\, W_{\id}(f)}
$$
when $f:S_Y\longrightarrow Y$ is bounded and uniformly continuous (actually, a more general result holds \cite{Rodriguez-JMAA}). The above equality also holds, in some cases, when $g=\ixy$ and $f$ is uniformly continuous \cite{MarMerPay}. On the other hand, if the equality above holds for all bounded functions $f$, then $Y$ is uniformly smooth \cite{Rodriguez-PAMS}.

The main objective in this paper is to show (Theorem~\ref{main-theorem}) that the equality
\begin{equation*}%\label{eq:ecc-general}%\tag{$\star\star$}
V_g(f)=\ec\, \Wa_g(f)
\end{equation*}
holds for every $f,g\in \ell_\infty(\Gamma,Y)$. This is the content of section~\ref{sec:twonumericalranges}. Even in the case when $\Gamma=S_Y$ and $g=\id$, the result is interesting as it is not true replacing the approximated numerical range by the spatial numerical range (see section~\ref{sec:spatial}).

We deal in section~\ref{sec:spatial} with the relationship between the spatial and the approximate spatial numerical range. In this case, we look for conditions for which the equality
\begin{equation}\label{eq:Wa=overlineW}\tag{$\star$}
\Wa_g(f) =\overline{W_g(f)}
\end{equation}
holds true. It has been shown in \cite{Ardalani} that \eqref{eq:Wa=overlineW} holds for $\Gamma=S_Y$ and $g=\id$ when $f$ is (the restriction to $S_Y$ of) a bounded linear operator. We extend this result to bounded uniformly continuous functions $f$. Besides, we show tht this cannot be extended to arbitrary bounded functions unless the space $Y$ is uniformly smooth, in which case one actually has that \eqref{eq:Wa=overlineW} works for every non-empty set $\Gamma$ and for all bounded functions $f$, which the only requirement that $g(\Gamma)\subset S_Y$. On the other hand, we show that if $\Gamma$ is a compact topological space, then \eqref{eq:Wa=overlineW} holds for all continuous $f$ and $g$. In particular, one has the validity of \eqref{eq:Wa=overlineW} when $X$ is a finite-dimensional subspace of $Y$, $\Gamma=S_X$, $g=\ixy$ and $f$ is continuous. We also show that this result cannot be extended to any infinite-dimensional $X$.

Let us finish the introduction saying that the study of the approximate spatial numerical range allows to better understand when the equality
$$
V_g(f)= \overline{\ec\, W (\ell_\infty(\Gamma,Y),g,f)}
$$ holds, as it can be deduced from  equality \eqref{eq:Wa=overlineW}, and this one only involves ``spatial type'' numerical ranges and does not force to work with the dual of $\ell_\infty(\Gamma,Y)$.

\section{Relation between the intrinsic numerical range and the approximated spatial numerical range} \label{sec:twonumericalranges}

Our main goal in this section is to prove the following relation between these two numerical ranges.

\begin{theorem}\label{main-theorem}
Let $Y$ be a Banach space, let $\Gamma$ be a non-empty set and consider a fixed $g\in \ell_\infty(\Gamma,Y)$ with $\|g\|=1$.
Then
$$
V_g(f)=\ec\, \Wa_g(f)
$$
for every $f\in \ell_\infty(\Gamma,Y)$.
\end{theorem}

In order to prove the theorem, we need some lemmata.

\begin{lemma}\label{lemma-easy} Let $Y$ be a Banach space and $\Gamma$ be a non-empty set. Then
\begin{align*}
V_g(f) = \bigcap_{\eps>0}\overline{\bigl\{\Phi(f)\, : \, \Phi\in \ell_\infty(\Gamma,Y)^*,\ \|\Phi\|=1,\  \re \Phi(g)>1-\eps\bigr\} }
\end{align*}
for every $f,g\in \ell_\infty(\Gamma,Y)$.
\end{lemma}

\begin{proof}
It can be easily proved using the Banach-Alaouglu theorem, but also follows from \cite[Fact~2.9.63]{Cabrera-Rodriguez} (which, actually, does not depend on the Banach-Alaouglu theorem).
\end{proof}

The proof of the next result is completely straightforward, so we omit it.

\begin{lemma}\label{lemma:decreasing-Wa} Let $Y$ be a Banach space and $\Gamma$ be a non-empty set. Given any sequence $\{\eps_n\}$ of positive numbers
decreasing to $0$, we have that
$$
\Wa_g(f):=\bigcap_{n\in \N}\overline{\bigl\{y^*(f(t))\,:\, y^*\in S_{Y^*},\, t\in \Gamma,\, \re y^*(g(t))>1-\eps_n\bigr\}}.
$$
for every $f,g\in \ell_\infty(\Gamma,Y)$.
\end{lemma}

\begin{lemma}\label{Javier}
Let $\{W_n\}$ be a decreasing sequence of compact subsets of $\K$, and define
$W=\bigcap_{n\in \N} W_n$. Then
$\displaystyle \sup \re W = \inf_{n\in \N} \sup \re W_n$.
\end{lemma}

\begin{proof}
That $\sup \re W \leq \inf_{n\in \N} \sup \re W_n$ is obvious as $W\subseteq W_n$ for every $n\in \N$. To get the reversed inequality, we first write
$t_n=\max \re W_n$ for every $n\in \N$, and observe that $\{t_n\}$ is a decreasing and bounded from bellow sequence, and write $t_0=\lim t_n$. On the one hand, observe that
$$
\inf_{n\in \N} \sup \re W_n = t_0.
$$
On the other hand, for every $q\in \N$, $t_{q+n}\in \re W_{q+n}\subseteq \re W_q$ for every $n\in \N$ so, taking limit in $n$, we get that $t_0\in \re W_q$. Therefore, $t_0\in \re W$ and so, $\sup \re W \geq t_0$.
\end{proof}

The next lemma is the key ingredient in the proof of Theorem~\ref{main-theorem}. It is an extension of a result of L.~Harris \cite{Harris}, proved for bounded uniformly continuous functions. We include the proof, which is an extension of the one given there, for the sake of completeness.

\begin{lemma}[\mbox{Extension of \cite[Lemma~1]{Harris}}]\label{lemma-Harris}
Let $Y$ be a Banach space and let $\Gamma$ be a non-empty set. Then,  given $f,g\in \ell_\infty(\Gamma,Y)$ with $\|g\|=1$ and $\Phi\in \ell_\infty(\Gamma,Y)^*$ with $\|\Phi\|=\Phi(g)=1$, there exist $t\in \Gamma$ and $y^*\in S_{Y^*}$ such that
$$
\re \Phi(f) < \re y^*(f(t))  + \eps \qquad \text{and} \qquad \re y^*(g(t)) > 1-\eps.
$$
\end{lemma}

\begin{proof}
Consider the subset of $S_{\ell_\infty(\Gamma,Y)^*}$ given by
$$
\Upsilon=\bigl\{y^*\otimes \delta_t\,:\, y^*\in S_{Y^*},\, t\in \Gamma \bigr\},
$$
where $[y^*\otimes \delta_t](h)=y^*(h(t))$ for every $h\in \ell_\infty(\Gamma,Y)$, every $t\in \Gamma$ and every $y^*\in S_{Y^*}$. Then, the unit ball of $\ell_\infty(\Gamma,Y)^*$ is the weak$^*$-closed convex hull of $\Upsilon$ (indeed, this follows from the immediate fact that
$$
\|h\|=\sup \bigl\{\re [y^*\otimes \delta_t](h)\,:\, y^*\otimes \delta_t\in \Upsilon\bigr\}
$$
for every $h\in \ell_\infty(\Gamma,Y)$). Therefore, for $0<\eps'<\eps$ satisfying $2\|f\|\eps'\leq \eps$, we may find $y_1^*,\ldots,y_n^*\in S_{Y^*}$, $t_1,\ldots,t_n\in \Gamma$, $\alpha_1,\ldots,\alpha_n\in [0,1]$ with $\sum_{k=1}^n \alpha_k=1$ such that
\begin{align*}
\sum_{k=1}^n \alpha_k \re y_k^*(f(t_k)) &> \re \Phi(f)-\eps/2 \\
\intertext{and}
\sum_{k=1}^n \alpha_k \re y_k^*(g(t_k)) &> 1-(\eps')^2.
\end{align*}
Now, consider
$$
J=\bigl\{k\,:\,1\leq k \leq n,\, \re y^*_k(g(t_k))>1-\eps'\bigr\}
$$
and let $K=\{1,\ldots,n\}\setminus J$. We have that
$$
1-(\eps')^2 < \sum_{k=1}^n \alpha_k \re y_k^*(g(t_k)) \leq \sum_{k\in J} \alpha_k + \sum_{k\in K} \alpha_k (1-\eps') = 1 - \eps'\sum_{k\in K}\alpha_k,
$$
from which we deduce that
$$
\sum_{k\in K}\alpha_k < \eps'.
$$
Now, we have that
\begin{align*}
\re \Phi(f) - \eps/2 & < \sum_{k=1}^n \alpha_k \re y_k^*(f(t_k)) \\ & \leq \sum_{k\in J} \alpha_k\re y_k^*(f(t_k)) + \|f\|\sum_{k\in K} \alpha_k \\ & < \sum_{k\in J} \alpha_k\re y_k^*(f(t_k)) + \eps/2.
\end{align*}
Therefore,
$$
\sum_{k\in J} \alpha_k\re y_k^*(f(t_k)) > \re \Phi(f) - \eps,
$$
and an obvious convexity argument provides the existence of $k\in J$ such that
$$
\re y_k^*(f(t_k))> \re \Phi(f) -\eps.
$$
On the other hand, as $k\in J$, we have
$$
\re y_k^*(g(t_k)) > 1-\eps'>1-\eps,
$$
finishing the proof.
\end{proof}

\begin{proof}[Proof of Theorem~\ref{main-theorem}]
We start with the inclusion ``$\supseteq$''. Fix $\lambda\in \Wa_g(f)$. For every $\eps>0$ and $\delta>0$, there exist $y^*\in S_{Y^*}$ and $t\in \Gamma$ such that
$$
\re y^*(g(t))>1-\eps \qquad \text{and} \qquad |\lambda - y^*(f(t))|<\delta.
$$
Then, defining $\Phi_{\eps,\delta}\in \ell_\infty(\Gamma,Y)^*$ by $\Phi_{\eps,\delta}(h)=y^*(h(t))$ for every $h\in \ell_\infty(\Gamma,Y)$, we have that $\|\Phi_{\eps,\delta}\|=1$, $\re \Phi_{\eps,\delta}(g)>1-\eps$ and $|\lambda - \Phi_{\eps,\delta}|<\delta$. Moving $\delta \downarrow 0$, it follows that
$$
\lambda\in \overline{\bigl\{\Phi(f)\, : \, \Phi\in \ell_\infty(\Gamma,Y)^*,\ \|\Phi\|=1,\ \re \Phi(g)>1-\eps\bigr\} }.
$$
Therefore, $\lambda \in V_g(f)$ using Lemma~\ref{lemma-easy}. Finally, $V_g(f)$ is convex, so the desired inclusion follows.

Let us prove the reversed inclusion. It is enough to prove that
$$
\sup \re V_g(f) \leq \sup \re \Wa_g(f)
$$
for every $f\in \ell_\infty(\Gamma,Y)$. Indeed, it is straightforward to show that for every $\theta\in \K$ with $|\theta|=1$ and every $f\in \ell_\infty(\Gamma,Y)$, we have
$$
V_g(\theta f)=\theta \, V_g(f)
\quad \text{ and } \quad
\Wa_g(\theta f)= \theta \, \Wa_g(f).
$$
From this, and the fact that both ranges are closed, the result follows (see \cite[Proposition~1]{Harris}, for instance).

Then, fix $f\in \ell_\infty(\Gamma,Y)$. Consider $\Phi\in \ell_\infty(\Gamma,Y)^*$ such that $\|\Phi\|=\Phi(g)=1$. For every $n\in \N$, write
$$
W_n=
\overline{\bigl\{y^*(f(t))\,:\, y^*\in S_{Y^*},\, t\in \Gamma,\, \re y^*(g(t))>1-\tfrac1n\bigr\}}
$$
and observe (Lemma~\ref{lemma:decreasing-Wa}) that
$$
\Wa_g(f)=\bigcap\nolimits_{n\in \N} W_n.
$$
Next, for each $n\in \N$ we use Lemma~\ref{lemma-Harris} to get $y_n^*\in S_{Y^*}$ and $t_n\in \Gamma$ such that
$$
\re \Phi(f) < \re y_n^*(f(t_n))  + \tfrac{1}{n} \qquad \text{and} \qquad \re y_n^*(g(t_n)) > 1-\tfrac{1}{n}.
$$
Therefore,
$$
\re \Phi(f) < \sup \re W_n + \tfrac{1}{n},
$$
for every $n\in \N$, and so
$$
\re \Phi(f) \leq \inf_{n\in \N} \sup \re W_n.
$$
Now, Lemma~\ref{Javier} shows that $\re \Phi(f)\leq \sup \re \Wa_g(f)$, and the arbitrariness of $\Phi$ gives $$
\sup \re V_g(f) \leq \sup \re \Wa_g(f),
$$
as desired.
\end{proof}

\section{Relation between the spatial numerical range and the approximate spatial numerical range}\label{sec:spatial}
Here, we would like to study conditions for which the closure of the spatial numerical range coincides with the approximate spatial numerical range. When $g=\id$, it is shown in \cite[Lemma~2.4]{Ardalani} that $\overline{W(T)}=\Wa_{\id} (T)$ for every bounded linear operator $T\in L(Y)$. The result easily extends to bounded uniformly continuous functions. We include a proof, extension of the one given in \cite{Ardalani} and consequence of the Bishop-Phelps-Bollob\'{a}s theorem, for the sake of completeness.

\begin{proposition}\label{prop:spatial-BPB}
Let $Y$ be a Banach space. Then,
$$
    \overline{W_{\id}(f)}= \Wa_{\id} (f)
$$
for every $f:S_Y\longrightarrow Y$ bounded and uniformly continuous.
\end{proposition}

\begin{proof}
We have that $W_{\id}(f)\subseteq \Wa_{\id} (f)$, and the second set is closed, so one inclusion is clear. Let us prove the more intriguing reversed inclusion. Fix $\mu\in \Wa_{\id}(f)$ and $\eps>0$. We use the uniform continuity of $f$ to find $0<\gamma<\eps$ such that
$$
\|f(y_1)-f(y_2)\|<\eps \quad \text{whenever $y_1,y_2\in S_Y,\ \|y_1-y_2\|<\gamma$.}
$$
Next, we take $y_0\in S_Y$ and $y_0^*\in S_{Y^*}$ such that
$$
\re y_0^*(y_0)>1-\gamma^2/2 \quad \text{and} \quad \bigl|\mu - y_0^*(f(y_0))\bigr|<\eps.
$$
Then, by the Bishop-Phelps-Bollob\'{a}s theorem \cite{Bollobas} (see \cite{C-K-M-M-R} for this version) there exist $y\in S_Y$, $y^*\in S_{X^*}$ such that
$$
\|y-y_0\|<\gamma,\quad \|y^*-y_0^*\|<\gamma<\eps,\quad \text{and} \quad y^*(y)=1.
$$
Now,
\begin{align*}
\bigl|\mu- y^*(f(y))\bigr| & \leq |\mu - y_0^*(f(y_0))| + |y_0^*(f(y_0))-y_0^*(f(y))| + |y_0^*(f(y))-y^*(f(y))| \\ & < \eps + \|f(y_0)-f(y)\| + \|y^*-y_0^*\|\|f\| \\ &<\eps + \eps + \eps\|f\|.
\end{align*}
Moving $\eps\downarrow 0$, we get $\mu\in \overline{W_{\id}(f)}$, as desired.
\end{proof}

\begin{remark}
The proof of the result above can be easily extended to the case when $X$ is a closed subspace of a Banach space $Y$, $g=\ixy$, and the pair $(X,Y)$ has a property introduced in \cite[\S 4]{MarMerPay} which plays the rolle of the Bishop-Phelps-Bollob\'{a}s theorem. However, the most interesting examples are covered by our results in the rest of the section, namely Corollary~\ref{cor:findim-clW=widetildeW} and Proposition~\ref{prop-spatial-unifsmooth-sufficient}.
\end{remark}

We cannot extend Proposition~\ref{prop:spatial-BPB} above to all bounded functions, as the following result shows.

\begin{proposition}\label{prop:uniformlysmooth-reversed}
Let $Y$ be a Banach space. If the equality $\overline{W_{\id}(f)} = \Wa_{\id}(f)$ holds for every $f\in \ell_\infty(S_Y,Y)$, then $Y$ is uniformly smooth.
\end{proposition}

\begin{proof}
By Theorem~\ref{main-theorem} and the assumption, we have that
$$
\overline{\ec\, W_{\id}(f)} =V_{\id}(f)
$$
for every $f\in \ell_\infty(S_Y,Y)$, and then \cite[Theorem~5]{Rodriguez-PAMS} gives that $Y$ is uniformly smooth.
\end{proof}

The converse of this result is also true, actually more, as we will see in Proposition~\ref{prop-spatial-unifsmooth-sufficient}.

Our next goal here is to study the relation between the spatial and the approximate spatial numerical ranges when $g=\ixy$. We will get positive results when $X$ is finite-dimensional and when $Y$ is uniformly smooth, consequences of deeper results. Let us start with the case when $X$ is finite-dimensional.

\begin{proposition}
Let $\Gamma$ be a compact topological space and let $Y$ be a Banach space. Then, given a continuous function $g:\Gamma\longrightarrow Y$ with norm one, we have
$$
    W_g(f)= \Wa_g (f)
$$
for every $f:\Gamma\longrightarrow Y$ continuous.
\end{proposition}

\begin{proof}
Only the inclusion ``$\supseteq$'' has to be proved. Fix $\mu\in \Wa_g (f)$ and consider two nets $(t_\lambda)_\lambda\in\Lambda$ and $(y^*_\lambda)_\lambda\in\Lambda$ such that
$$
y^*_\lambda(g(t_\lambda))\longrightarrow 1 \qquad \text{and} \qquad
y_\lambda^*(f(t_\lambda))\longrightarrow \mu.
$$
As $\Gamma$ is compact, we may and do suppose that the net $t_\lambda$ converges to $t_0\in \Gamma$. Also, as $B_{Y^*}$ is weak$^*$-compact, we may and do suppose that $(y^*_\lambda)$ converges to $y_0^*\in B_{Y^*}$ in the weak$^*$-topology. Now, it follows from the continuity (in norm) of the functions $g$ and $f$ that
$$
y_0^*(g(t_0))=1 \qquad \text{and} \qquad y_0^*(f(t_0))=\mu.
$$
Therefore, $\mu\in W_g(f)$, as desired.
\end{proof}

As a particular case, we can take $g=\ixy$ when $X$ is finite-dimensional.

\begin{corollary}\label{cor:findim-clW=widetildeW}
Let $X$ be a finite-dimensional Banach space. Then for every Banach space $Y$ containing $X$ isometrically and for every continuous function $f:S_X \longrightarrow Y$, we have
$$
W_{\ixy}(f)= \Wa_{\ixy} (f).
$$
\end{corollary}

This result cannot be extended to any infinite-dimensional space $X$, as the following example shows.

\begin{example}\label{example:infdim-no-W-Wa}
For every infinite-dimensional Banach space $X$, there exist a Banach space $Y$ containing $X$ as a hyperplane isometrically, and a bounded linear operator $T:X\longrightarrow Y$ such that
$$
\overline{W_{\ixy}(T)} \neq \Wa_{\ixy}(T).
$$
\end{example}

\begin{proof}
By \cite[Remark~2.3]{MarMerPay}, there exists $Y$ and $T$ as in the statement such that $\overline{\ec\, W_{\ixy}(T)} \neq V_{\ixy}(T)$. By Theorem~\ref{main-theorem}, the same example works.
\end{proof}

We now deal with the case when $Y$ is uniformly smooth, presenting the following very general result. Observe that the only requirement on $g$ is that its image falls into $S_Y$. Recall that a Banach space $Y$ is \emph{uniformly smooth} if whenever $(y_n^*)$ and $(z_n^*)$ are sequences in $B_{X^*}$ such that $\|y_n^* + z_n^*\|\longrightarrow 2$, it follows that $\|y_n^*-z_n^*\|\longrightarrow 0$.

\begin{proposition}\label{prop-spatial-unifsmooth-sufficient}
Let $Y$ be a uniformly smooth Banach space, let $\Gamma$ be a non-empty set and $g:\Gamma\longrightarrow S_Y$. Then,
    $$
    \overline{W_g(f)}= \Wa_g(f)
    $$
for every $f\in \ell_\infty(\Gamma,Y)$.
\end{proposition}

\begin{proof}
Only the inclusion ``$\supseteq$'' has to be proved. Fix $\lambda \in \Wa_g(f)$ and consider two sequences $(t_n)$ in $\Gamma$ and $(y_n^*)$ in $S_{Y^*}$ such that
$$
y_n^*(g(t_n))\longrightarrow 1 \qquad \text{and} \qquad y_n^*(f(t_n))\longrightarrow \lambda.
$$
For every $n\in \N$, we take $z_n^*\in S_{Y^*}$ such that $z_n^*(g(t_n))=1$. As we have that
$\|y_n^*+z_n^*\|\longrightarrow 2$,
the uniform smoothness of $Y$ gives that
$
\|y_n^*-z_n^*\|\longrightarrow 0.
$
Now, we have that
\begin{align*}
\bigl|z_n^*(f(t_n)) - \lambda \bigr| &\leq \bigl|y_n^*(f(t_n)) - \lambda \bigr| + \bigl|y_n^*(f(t_n)) - z_n^*(f(t_n)) \bigr| \\ & \leq \bigl|y_n^*(f(t_n)) - \lambda \bigr| + \|y_n^* - z_n^*\|\|f\| \longrightarrow 0.
\end{align*}
As $z_n^*(g(t_n))=1$, we get that $\lambda\in \overline{W_g(f)}$.
\end{proof}

In particular, we get the following.

\begin{corollary}
Let $Y$ be a uniformly smooth Banach space and let $X$ be a closed subspace of $Y$. Then,
$$
    \overline{W_{\ixy}(f)}= \Wa_{\ixy} (f)
    $$
for every $f\in \ell_\infty(S_X,Y)$.
\end{corollary}

That uniform convexity is essential in the result above is shown by Proposition~\ref{prop:uniformlysmooth-reversed}.


\begin{thebibliography}{99}

\bibitem{Ardalani} \textsc{M.~A.~Ardalani}, Numerical index with respect to an operator, \emph{Studia Math.} \textbf{224} (2014), 165--171.

\bibitem{Bollobas} \textsc{B.~Bollob\'{a}s}, An extension to the theorem of Bishop and Phelps, \emph{Bull. London Math. Soc.} \textbf{2} (1970), 181--182.

\bibitem{B-K} \textsc{H.~F.~Bohnenblust and S.~Karlin},
Geometrical properties of the unit sphere in Banach algebras,
\emph{Ann. of Math.} \textbf{62} (1955), 217--229.

\bibitem{B-D1} \textsc{F.~F.~Bonsall and J.~Duncan},
\emph{Numerical Ranges of operators on normed spaces and of elements
of normed algebras}, London Math. Soc. Lecture Note Series
\textbf{2}, Cambridge University Press, 1971.

\bibitem{B-D2} \textsc{F.~F.~Bonsall and J.~Duncan},
\emph{Numerical Ranges II}, London Math. Soc. Lecture Note Series
\textbf{10}, Cambridge University Press, 1973.

\bibitem{Cabrera-Rodriguez} \textsc{M.~Cabrera and A.~Rodr\'{\i}guez Palacios}, \emph{Non-associative normed algebras}, volume 1: the Vidav-Palmer and Gelfand-Naimark Theorems, Encyclopedia of Mathematics and Its Applications \textbf{154}, Cambridge Univesity press, 2014.

\bibitem{C-K-M-M-R}
 \textsc{M.~Chica, V.~Kadets, M.~Mart{\'{\i}}n, S.~Moreno-Pulido,
  and F.~Rambla-Barreno}. Bishop-{P}helps-{B}ollob\'as
  moduli of a {B}anach space, \emph{J. Math. Anal. Appl.} \textbf{412} (2014),
  no.~2, 697--719.

\bibitem{Harris} \textsc{L.~A.~Harris}, The numerical range of functions and best approximations, \emph{Proc. Camb. Phil. Soc.} \textbf{76} (1974), 133--141.

\bibitem{Harris-holomorphic} \textsc{L.~A.~Harris}, The numerical range of holomorphic functions in Banach spaces, \emph{Amer. J. Math.} \textbf{43} (1971), 1005--1019.

\bibitem{MarMerPay} \textsc{M.~Mart\'{\i}n, J.~Mer\'{\i} and R.~Pay\'{a}}, On the intrinsic and the spatial numerical range, \emph{J. Math. Anal. Appl.} \textbf{318} (2006), 175--189.

\bibitem{M-M-P-R} \textsc{J.~Mart\'{\i}nez, J.~F.~Mena, R.~Pay\'{a} and
A.~Rodr\'{\i}guez-Palacios}, An approach to numerical ranges without Banach algebra
theory, \emph{Illinois J. Math.} \textbf{29} (1985), no. 4, 609--626.

\bibitem{Rodriguez-PAMS} \textsc{A.~Rodr\'{\i}guez Palacios}, A numerical ranges characterization of uniformly smooth Banach spaces, \emph{Proc. Amer. Math. Soc.}, \textbf{129} (2001), 815--821.

\bibitem{Rodriguez-JMAA} \textsc{A.~Rodr\'{\i}guez Palacios}, Numerical ranges of uniformly continuous functions on the unit sphere of a Banach space, \emph{J. Math. Anal. Appl.}, \textbf{297} (2004), 472--476.

\end{thebibliography}
\end{document}